\documentclass[pdftex,preprint]{imsart}
\RequirePackage[OT1]{fontenc}
\RequirePackage{amsthm,amsmath,natbib}
\RequirePackage[colorlinks,citecolor=blue,urlcolor=blue]{hyperref}
\RequirePackage{hypernat}
\usepackage{amssymb,ulem}
\usepackage[counterclockwise]{rotating}
\usepackage[T1]{fontenc}
\usepackage[latin1]{inputenc}
\usepackage{graphicx}
\usepackage{mathpazo}
\usepackage{pdfsync}

\def\E{\mathbb E}
\def\R{\mathbb R}

\def\N{\mathbb N}
\def\P{\mathbb P}

\def\Z{\mathbb Z}

\def\bfB{\mathbf B}

\def\bfV{\mathbf V}

\def\bfP{\mathbf P}
\def\bfQ{\mathbf Q}

\def\cF{{\mathcal F}}

\newtheorem{Def}{Definition}
\newtheorem{Prop}{Proposition}
\newtheorem{lem}{Lemma}
\newtheorem{Theo}{Theorem}

\newtheorem{rem}{Remark}

\begin{document}

\begin{frontmatter}
\title{Parametric inference and forecasting for continuously invertible volatility models}
\runtitle{Inference under invertibility}
\begin{aug}
 \author{\fnms{Olivier} \snm{Wintenberger} \ead[label=e1]{wintenberger@ceremade.dauphine.fr}}
 \affiliation{Centre De Recherche en Math\'ematiques de la D\'ecision,  UMR CNRS 7534
         \\\small Universit\'e de Paris-Dauphine}
{\footnotesize
\centerline{  Universit\'e de Paris-Dauphine}
}
\printead*{e1}
\\

\and
\\

\author{\fnms{Sixiang} \snm{Cai}\ead[label=e2]{sixiang.cai@u-cergy.fr}}
\affiliation{Universit\'e de Cergy-Pontoise, D\'epartement de Math\'ematiques , UMR CNRS 8088}
{\footnotesize
\centerline{ Universit\'e de Cergy-Pontoise}}
\printead*{e2}
\date{}

\runauthor{O. Wintenberger and S. Cai}
\end{aug}

\begin{abstract}
We introduce the  notion of continuous  invertibility on a compact set for volatility models driven by a Stochastic Recurrence Equation (SRE). We prove in this context the strong consistency and the asymptotic normality of  the $M$-estimator associated with the Quasi-Likelihood criteria. We recover known results 
on univariate and multivariate GARCH type models where the estimator coincides with the classical QMLE. In EGARCH type models. our approach gives a strongly consistence  and  asymptotically normal estimator when the limiting covariance matrix exists. We provide a necessary and sufficient condition for the existence of this limiting covariance matrix in the EGARCH(1,1) model introduced in \cite{Nelson1991}. We exhibit for the first time sufficient conditions for the asymptotic normality of the estimation procedure used in practice since \cite{Nelson1991}.\end{abstract}

\begin{keyword}[class=AMS]
\kwd[Primary ]{62F12}
\kwd[; secondary ]{60H25, 62F10, 62M20, 62M10, 91B84}
\end{keyword}

\begin{keyword}
\kwd{Invertibility, volatility models, parametric estimation, strong consistency, asymptotic normality, asymmetric GARCH, exponential GARCH, 
stochastic recurrence equation, stationarity}
\end{keyword}

\end{frontmatter}

\section{Introduction}\label{sec:intro}

Since the seminal paper of \citet{Engle1982} and  \citet{Bollerslev1986}, the General Autoregressive Conditional Heteroskedasticity  (GARCH) type models have been successfully applied to volatility modeling.   \citet{Nelson1991} is the first attempt to introduce non linearity into volatility models with the Exponential-GARCH(1,1) type models. Since then, many other volatility models have been introduced: APGARCH of \citet{Ding1993}, GJR-GARCH of \citet{GJR1993}, TGARCH of \citet{Zakoian1994}, etc. Non linear volatility models have been used extensively in empirical researches (see \citet{BJ2006} among many others) and financial industry.  Not surprisingly, theoretical investigations of EGARCH has attracted constant attention, see \citet{HCTM2002}, \citet{H2010} and \citet{RR2009}. However, the validity of the estimation procedures used empirically in \cite{Nelson1991} was not proved. Our study provides the first satisfactory answer to this open question for non linear volatility models including the EGARCH(1,1) model. We give sufficient conditions for the estimator to be strongly consistent and asymptotically normal.  Our approach is based on the natural notion of continuous invertibility that we introduce in a very general setting and thus will be applied in other models in future works.\\

Consider a general volatility model of the form $ X_{t} =\Sigma_{t}^{1/2}\cdot Z_{t}$ where $\Sigma_{t}$ is the volatility and where the innovations $Z_t$ are normalized, centered independent identical distributed (iid) random vectors. The natural filtration $ \cF_{t} $ is generated by the past innovations $(Z_{t},Z_{t-1},\ldots)$. It is assumed that a transformation of the volatility satisfies some (possibly non-linear) SRE, i.e. there exist a function $h$  and some $\cF_{t-1}$ measurable random function $\psi_{t}$ such that the following relation \begin{equation}\label{eq:gmod}
(h(\Sigma_{k}))_{k\leq t}  = \psi_{t}((h(\Sigma_{k}))_{k\leq t-1},\theta_{0})
\end{equation} holds.  It is the case of all classical models of GARCH and EGARCH types, extensions of the  simplest GARCH(1,1) and EGARCH(1,1) univariate models. Let us illustrate our propose in this introduction on these  models (in the univariate case, the volatility is denoted $\sigma^2_t$):
\begin{align}
\mbox{GARCH(1,1):}& \qquad\sigma_t^2=\alpha_0+\beta_0 \sigma_{t-1}^2+\gamma_0 X_{t-1}^2,\label{eq:garch}\\
\mbox{EGARCH(1,1):}&\qquad  \log(\sigma_t^2)=\alpha_0+\beta_0 \log(\sigma_{t-1}^2)+(\gamma_0 Z_{t-1}+\delta_0|Z_{t-1}|).\label{eq:egarch}
\end{align}
One can rewrite \eqref{eq:garch} as an SRE driven by the innovations: $\sigma_t^2=\alpha_0+(\beta_0+\gamma_0Z_{t-1}^2) \sigma_{t-1}^2$, i.e. $\psi_t(x,\theta)=\alpha +(\beta +\gamma Z_{t-1}^2) x$. This SRE  is used by \cite{Nelson1990} to obtain the Lyapunov condition $\E[\log(\beta_0+\gamma_0 Z_0^2)]<0$, necessary and sufficient for the stationarity. In general, the functional process $(\psi_t)$ driving the SRE \eqref{eq:gmod} is assumed to be a stationary ergodic process of Lipschitz functions. Such a SRE is said to be convergent when its solution is unique, non anticipative (i.e. function of $\cF_{t}$ at any time $t$) and its law does not depend on the initial values. This last property, also called the "stability" of the SRE, ensures  the existence of the stationary process $(X_t)$. It coincides with the ergodicity for Markov chains. Sufficient conditions   (also necessary in the linear Markov case)  for the convergence are the negativity of a Lyapunov coefficient and the existence of logarithmic moments, see \citet{Elton1990}, \citet{Bougerol1992} and \citet{Bougerol1993}.\\

Assume that the model have a non anticipative, stationary solution with invertible volatility matrices $\Sigma_t$. Using the relation $Z_t=\Sigma_t^{-1}\cdot X_t$ in $\psi_t$, it is possible to study a new SRE driven by the observations $X_t$
\begin{equation}\label{SREinftrue}
(h(\Sigma_k))_{k\le t}=\phi_t((h(\Sigma_k))_{k\le t-1},\theta_0).
\end{equation}
Here $\phi_t$ is an ergodic stationary process  generated by $(\mathcal G_{t-1})$, the sigma-field of the past values $(X_{t-1},X_{t-2},\cdots)$. The convergence of this new SRE is closely related with the notion of invertibility (i.e. the existence of a non linear AR representation of the observations)  see \citet{Granger1978}, \citet{Tong1993}, \citet{Straumann2005} and \citet{SM2006}. As previously, sufficient conditions for the convergence are the negativity of a Lyapunov coefficient and the existence of logarithmic moments; These conditions are expressed this time on $\phi_1$ and not on $\psi_1$. For instance, we obtain from \eqref{eq:garch} that $\phi_t(x,\theta )=\alpha+\beta x+\gamma X_{t-1}^2$ and the   GARCH(1,1) model is invertible as soon as $0\le \beta_0<1$. For the EGARCH(1,1) model, we obtain from \eqref{eq:egarch} that
\begin{equation}\label{eq:egarchinv}
\phi_t(x,\theta)=\alpha+\beta x+(\gamma  X_{t-1}+\delta |X_{t-1}|)\exp(-x/2).
\end{equation}
This SRE is less stable than in the GARCH(1,1) model because of the exponential function that explodes for large values. However, a sufficient condition for invertibility is given in \citet{Straumann2005} under restrictions on the parameters $(\alpha_0,\beta_0,\gamma_0,\delta_0)$, see  \eqref{eq:condtheta2} below.\\

In practice the parameter $\theta_0$ is unknown and is estimated using the SRE as follow. An approximation of the SRE $(h(g_k))_{k\le t}=\phi_t((h(g_k))_{k\le t-1},\theta)$ generates recursively a forecast $\hat g_t(\theta)$ of  the volatility $\Sigma_t$ using only the past observations and the  model at the point $\theta$. Using Quasi-LIKelihood (QLIK) criteria to quantify the error of the volatility forecasting:
\begin{equation}
 n\hat{S}_{n}(\theta)  =  \sum_{t=1}^{n}\hat{s}_{t}(\theta) 
  =  \sum_{t=1}^{n}2^{-1}\left(X_{t}^{T}\ell(\hat{g}_{t}(\theta))^{-1}X_{t}+\log(\det(\ell(\hat{g}_{t}(\theta)))\right)
\label{eq:hat_L}
\end{equation}
the best forecast corresponds to the M-estimator $\hat{\theta}_{n}$ satisfying
\begin{equation}
\hat{\theta}_{n}=\mbox{argmin}_{\theta\in\Theta}\hat{S}_{n}(\theta).
\label{eq:QMLE_hat_theta}
\end{equation}
Under conditions ensuring the asymptotic regularity of $\hat S_n(\cdot)$, $\hat\theta_n$ is a good candidate for estimating $\theta_0$. By construction, the asymptotic regularity  of $\hat S_n(\cdot)$ depends on the regularity of $g_t(\cdot)$ and on the stability of the corresponding SRE. As $\theta_0$ is unknown, \citet{Straumann2005} imposed the uniform invertibility over the compact set $\Theta$: The SRE is assumed to be stable in the Banach space of continuous function on $\Theta$ (with the sup-norm). The consistency and the asymptotic normality of $\hat\theta_n$ follows, see \cite{Straumann2005} and \cite{SM2006}. However, the notion of uniform invertibility is too restrictive; the asymptotic normality of the EGARCH(1,1) model is proved in the degenerate case $\beta_0=0$ only; there $\log(\sigma_t^2)=\alpha_0+(\gamma_0 Z_{t-1}+\delta_0|Z_{t-1}|)$ which is not realistic.\\

In this paper, we introduce the  notion of continuous  invertibility  and apply it successfully to the EGARCH($1,1$) model: the asymptotic normality of the procedure described above and used since \cite{Nelson1990} is proved for the first time. The notion of continuous invertibility is a very natural one: the SRE is assumed to be stable at each point $\theta\in\Theta$ and the functional solution $(g_t(\cdot))$ is assumed to be continuous. By definition, uniform invertibility implies the weaker notion of continuous invertibility. We give sufficient conditions {\bf (CI)} for continuous invertibility and check it on volatility model that are non necessarily uniformly invertible. We prove under {\bf (CI)} and the identifiability of the model that $\hat\theta_n$ is a strongly consistent estimator of $\theta_0$  and that the natural forecast $\hat g_t(\hat\theta_t)$ of the volatility $\Sigma_t$ is also strongly consistent. We prove that  $\hat\theta_n$ is  asymptotically normal if moreover the limiting variance exists. The proofs are very general and also valid in the multidimensional case. In absence of uniform invertibility, we use under {\bf (CI)} new arguments following the ones of \citet{Jeantheau1993} for the strong consistency and of \citet{BW2009} for the asymptotic normality. One crucial step in the proof is Theorem \ref{prop:log} below which asserts the logarithmic moments properties of solutions of SRE. This results, by its generality, is of independent interest for the study of  probabilistic properties of  solutions of SRE (also called Iterated Random Functions).\\

The  commonly used statistical procedure described above is only valid under {\bf (CI)}. Fortunately condition {\bf (CI)} is automatically satisfied for all invertible GARCH and EGARCH type models known by the authors. The GARCH(1,1) model satisfied {\bf (CI)} on all compact sets of $ [0,\infty[^3$ satisfying  $\beta<1$ where it is also uniformly invertible. The EGARCH(1,1) model satisfies {\bf (CI)} on all compact sets  of points that satisfy the invertibility condition \eqref{eq:condtheta}. It is not uniformly invertible there but the statistical inference is still valid. Applying our approach to other models, we recover the results of  \citet{BHK2003} and \citet{Francq2004} for GARCH(p,q) models, we recover the results of  \citet{Francq2011} for CCC-GARCH(p,q) models  and for AGARCH(p,q) models we refine the results of \citet{SM2006}. On the contrary, it is shown in \citet{Sorokin2011}  that forecasting the volatility with SRE may be inconsistent when the SRE is unstable. Thus, if the model is not continuously invertible on $\Theta$ the minimization \eqref{eq:QMLE_hat_theta} is unstable and the estimation procedure is not valid. To sum up, one can think of the following "equivalences":

\begin{center}
\begin{tabular}{|lcl|}
\hline
 \vspace{-1mm} Stability of the SRE generated by & $\Longleftrightarrow$ & stationarity, ergodicity\\
 the innovations $(Z_t,Z_{t-1},\ldots)$ & (A) & and log-moments  \\\hline
 \vspace{-1mm} Stability of the SRE generated by & $\Longleftrightarrow$ & invertibility, forecasting\\
 the observations $(X_t,X_{t-1},\ldots)$ & (B) & and statistical inference\\
 \hline
\end{tabular}
\end{center}\

The equivalence (A) is crucial when studying existence of stationary solutions of volatility models. We want to emphasize the importance of the second equivalence (B) for  the volatility forecast and the statistical inference using the QLIK criteria. For non uniformly invertible models, it is crucial to infer the model only in the domain of invertibility otherwise the whole procedure can fail. The consequences of this work on empirical study is huge as the equivalence (B) has been negligible in most existing works, see \citet{WC} for the EGARCH(1,1) case (applications on other classical models are also in progress).  Finally, notice that the statistical inference of $\theta_0$ is possible without assuming {\bf (CI)}: it has been  done by \citet{Zaffaroni2009} using  Whittle's estimator.\\

An outline of the paper can be given as follows. In Section \ref{sec:ci}, we discuss the standard notions of invertibility and introduce the  continuous invertibility and its sufficient condition {\bf (CI)}. In Section \ref{sec:inf} our main results on the statistical inference based on the SRE are stated. We apply this results in some GARCH type models and in the EGARCH(1,1) model in Section \ref{sec:egarch}. The Appendix contains the technical computation of the necessary and sufficient condition for the existence of the asymptotic variance of $\hat\theta_n$ in the EGARCH(1,1) model.
 
\section{Continuously invertible volatility models}\label{sec:ci}
\subsection{The general volatility  model}
In  this paper,  $(Z_t)$ is a stationary  ergodic sequence of real vectors called the innovations. Let us denote $\mathcal F_t$ the filtration
generated by $(Z_t,Z_{t-1},\ldots)$. Consider the general volatility model $
X_{t} = \Sigma_{t}^{1/2}\cdot Z_{t}$ where \eqref{eq:gmod} is satisfied: $
(h(\Sigma_{k}))_{k\leq t}  = \psi_{t}((h(\Sigma_{k}))_{k\leq t-1},\theta_{0})$. The function $h$ is injective from the space of real matrices of size $k\times k$ to an auxiliary separable metric space $F$. The random function $\psi_{t}(\cdot,\theta_0)$ 
is a $\cF_{t-1}$ adapted random function from the space of the sequences of elements in the image of $h$ to itself. Let us denote $\ell$ the inverse of $h$ (from the image of $h$ to the space of real matrices of size $k\times k$) and call it the link function.

\subsection{Convergent SRE and stationarity}

A first question regarding this very general model is wether or not a stationary solution exists. 
As the sequence of the transformed volatilities $(h(\Sigma_k ))_{k\le t}$ is a solution of a fixed point problem, 
we recall the  following result due to \citet{Elton1990} and \citet{Bougerol1993}. Let $(E,d)$ be a complete separable metric space. 
A map $f: E\to E$ is a Lipschitz map if
$\Lambda(f)=\sup_{(x,y)\in E^2}d(f(x),f(y))/d(x,y)$
is finite. For any sequence of random element in $(E,d)$,  $(X_{t})$ is said to be exponential almost sure convergence to 0  
$X_{t}\xrightarrow{\mbox{e.a.s.}}0$ as $t\to \infty$ if for $X_{t} =o(e^{-Ct})$ a.s. for some $C>0$. 

\begin{Theo}\label{thm:Sol of SRE}\label{th:bo}
Let $(\Psi_{t})$ be a stationary ergodic sequence of Lipschitz maps from $E$ to $E$. Suppose that 
$\E[\log^+(d(\Psi_0(x),x))]<\infty$ for some $x\in E$, that $\E[\log^{+}\Lambda(\Psi_{0})]<\infty$ and that for some 
integer $r\geq1,$
$$
\E[\log\Lambda(\Psi_{0}^{(r)})]=\E[\log\Lambda(\Psi_{0}\circ\cdots\circ\Psi_{-r+1})]<0.
$$
Then the SRE $
X_{t}=\Psi_{t}(X_{t-1})$ for all  $t\in\Z$ is convergent: it admits a unique stationary solution $(Y_{t})_{t\in\Z}$ which is
ergodic and for any $y\in E$
$$
Y_{t}=\lim_{m\rightarrow\infty}\Psi_{t}\circ\cdots\circ\Psi_{t-m}(y),\quad t\in\Z.
$$
The 
$Y_{t}$ are measurable with respect to the $\sigma(\Psi_{t-k}, k\geq0)$ and
$$
d(\tilde{Y}_{t},Y_{t})\xrightarrow{\mbox{e.a.s.}}0,\quad t\rightarrow\infty
$$
 such that $\tilde{Y}_{t}=\Psi_t(\tilde{Y}_{t-1})$ for all $t>0$.
\end{Theo}
The sufficient Lyapunov assumptions $
\E[\log\Lambda(\Psi_{0}^{(r)})] <0$ is also necessary in the linear case, see \citet{Bougerol1992}.\\

The results of Theorem \ref{th:bo} does not guaranty any moment property on the stationary solution.   The following result ensures the existence of logarithmic moments of the stationary solution under natural assumptions. It extends the result of \citet{Alsmeyer2001} to a much more general context.
\begin{Theo}\label{prop:log}
Under the assumptions of Theorem \eqref{th:bo} and  $\E[(\log^+d(\Psi_0(x),x))^2]<\infty$ the unique stationary solution satisfies $\E[\log^+(d(Y_0,y))]<\infty$ for all $y\in E$.
\end{Theo}

\begin{proof}
The basic inequality $\log(1+y+z)\le\log(1+y)+\log(1+z)$ will be used several time. Notice that for any r.v. $X\ge 0$ we have the equivalence $\E[\log(1+X)]<\infty$ iff $\E[\log^+(X)]<\infty$. Thus $\E[\log(1+d(Y_0,y))]<\infty$ for all $y\in E$ iff $\E[\log(1+d(Y_0,x))]<\infty$ since $d(Y_0 ,y))\le  d(Y_0,x))+d(x,y)$. We denote $\Psi^{(-m)}=\Psi_0\circ\cdots\circ\Psi_{-m}(x)$, $w=d(y, \Psi^{(1-m)}(x))\ge 0$ and $z=\Lambda( \Psi^{(1-m)})d(\Psi_{-m}(x),x)\ge 0$. From the triangular inequality $d(x,\Psi^{(-m)}(x))\le w+z$ we obtain that 
\begin{align}
\nonumber\log(1+d(x,\Psi^{(-m)}(x)))\le \log(1+w+z)\le \log(1+ w) + \log(1+z).\label{eq:rec}
\end{align}
Eqn. 27 in \cite{Bougerol1992} asserts the existence of $0<\rho<1$ and  $\epsilon>0$ satisfying
$$
\overline\lim_{m\to\infty} \frac1m \log(\Lambda(\Psi^{(-m)}))\le \log(\rho)-\epsilon\qquad a.s.
$$
Thus there exists a r.v. $M\in \N^\ast$ such that $ \Lambda(\Psi^{(-m)})\le \rho^m$ for all $m\ge M$. 
Writing $v_m= \log(1+d(x, \Psi^{(-m)}(x)))$, for all $m\ge M$ we have:
$$
v_m\le v_{m-1}+ \log(1+ \rho^{m-1}d(\Psi_{-m}(x),x)).
$$
A straightforward recurrence leads to the following upper bound of all $(v_m)_{m\ge M}$
$$
v_m\le v_M+ \sum_{j=1}^{\infty}  \log(1+\rho^{j-1}d(\Psi_{-j-M}(x),x)).
$$
As $\log(1+d(x,Y_0))=\lim_{m\to\infty}v_m$ a.s., it remains to prove that the series in the upper bound is summable and that the upper bound is integrable to conclude by the  dominated convergence Theorem. Using the stationarity of $(v_m)$, we know that $\E[\log(1+\rho^{j-1}d(\Psi_{-j-M}(x),x))]=\E[\log(1+\rho^{j-1}d(\Psi_{0}(x),x))]$ does not depend on $M$ and $j$. Thus $\E[ v_M]=\E[\log(1+d(x,\Psi_0(x)))]<\infty$ by assumption. We conclude the proof by comparing the series with an integral:
\begin{align*}
\sum_{j\ge 1}\E[\log(1+\rho^{j-1}d(\Psi_{0}(x),x))]&\le \frac1{1-\rho}\int_0^1\frac{\E[\log(1+ud(\Psi_{0}(x),x))]}{u}du.\end{align*} 
Let us prove that his integral converges as soon as $\E[(\log^+d(\Psi_0(x),x))^2]<\infty$. Using that $\E[\log(1+ud(\Psi_{0}(x),x))]=\int_0^\infty \P(\log(1+ud(\Psi_{0}(x),x))\ge t) dt$ and denoting $v=(e^t-1)/u$ the integral becomes:
\begin{align*}\int_0^1\int_0^\infty\frac{\P(\log(1+ud(\Psi_{0}(x),x))\ge t)}{u}dtdu&=\int_0^1\int_0^\infty\frac{\P( d(\Psi_{0}(x),x)\ge v)}{1+uv}dvdu.
\end{align*}
Using Fubini's theorem and $\int_0^1(1+uv)^{-1}du=\log(v+1)/v$ for all $v\ge 0$ we get an upper bound in term of
\begin{align*} \int_0^\infty\frac{\log(1+v)\P( d(\Psi_{0}(x),x)\ge v)}{v}dv.
\end{align*}
This integral converges in $+\infty$ as 
\begin{align*} \int_0^\infty\frac{\log(1+v)\P( d(\Psi_{0}(x),x)\ge v)}{1+v}dv=\E[\log(1+d(x,\Psi(x))^2]
\end{align*}
and the desired result follows.\end{proof}

In order to apply Theorem \ref{thm:Sol of SRE} in our case, let us  denote by $E$  the separable metric space of the sequences of elements in the image of $h$. Equipped with the metric $\sum_{j\ge 1}2^{-j} d(x_j,y_j)/(1+d(x_j,y_j))$, the space $E$ is  complete. A sufficient condition for stationarity of $(X_t)$ is that the SRE driven by $(\psi_t)$ converges in $E$. It simply expresses as the Lyapunov condition $\E[\log\Lambda(\psi_{0}^{(r)})]<0$ for some integer $r\geq1$ and some logarithmic moments. This assumption of stationarity is sufficient but not optimal in many cases:
\begin{rem}\label{rem:dim}
The state space of the SRE \eqref{eq:gmod}, denoted $E$, in its most general form, is  a space of infinite sequences. However in all classical models we can find a lag $p$ such that  $(h(\Sigma_k))_{t-p+1\le k\le t}=\psi_t((h(\Sigma_k))_{t-p\le k\le t-1},\theta_0)$. The state space $E$ is now the finite product of $p$ spaces. It can be equipped by unbounded metrics such that $p^{-1}\sum_{j=1}^pd(x_j,y_j)$ or $\sqrt{\sum_{j=1}^pd^2(x_j,y_j)}$. The product metric has to be carefully chosen  as it changes the value of the Lipschitz coefficients of the $\phi_t$. Yet, even if the products spaces are embedded, the smallest possible lag $p$ in the SRE yields the sharpest Lyapunov condition. Finally, if $E$ has a finite dimension and if the condition of convergence of the SRE expresses in term of the top Lyapunov coefficient, one can choose any metric induced by any norm, see \citet{Bougerol1993} for details.
\end{rem}
In view of Remark \ref{rem:dim}, instead of choosing a specific metric space $(E,d)$ we prefer to work under the less explicit assumption
\begin{description}
\item[(ST)] The process $(X_t)$ satisfying \eqref{eq:gmod} exists. It is a stationary, non anticipative and ergodic process with finite logarithmic moments.
\end{description}
In view of Theorem \ref{prop:log}, it is reasonable to  require that the solution has finite logarithmic moments. It is   very useful when considering the invertibility of the general model, see Proposition \ref{prop:obinv} below.

\subsection{The invertibility and the observable invertibility}\label{subsec:inv}

Now that under {\bf (ST)} the process $(X_t)$ is stationary and ergodic, we investigate the question of invertibility of the general model \eqref{eq:gmod}. We want to emphasize that the classical notions of  invertibility are related with convergences of SRE and implied by Lyapunov conditions.  Following \citet{Tong1993}, we say that a volatility model is invertible if the  volatility can be expressed as a function of the past observed values:
\begin{Def}\label{def:inv}
The model is invertible if the sequence of the volatilities $(\Sigma_t )$ is adapted to the filtration 
$(\mathcal G_{t-1})$ generated by $(X_{t-1},X_{t-2},\cdots)$.
\end{Def}
It is natural to assume invertibility to be able to forecast the volatility and to ensure the so-called "predictability" of the model. This notion of invertibility is very weak and consists  in restricting the underlying filtration $(\mathcal F_{t})$ of the SRE to $(\mathcal G_{t-1})$. Indeed, under {\bf (ST)} the filtration $\mathcal G_{t}\subseteq \mathcal F_t$ is well defined. If the volatility matrices are invertible, using $Z_t=\Sigma_t^{-1}\cdot X_t$ in $\psi_t$ we can express \eqref{eq:gmod} as \eqref{SREinftrue}: $(h(\Sigma_k))_{k\le t}=\phi_t((h(\Sigma_k))_{k\le t-1},\theta_0)$, a SRE driven by the whole past of the observations. The sequence of random functions  $(\phi_t)$
is an ergodic and stationary process adapted to $(\mathcal G_{t-1})$. Using the sufficient conditions of convergence of SREs given in Theorem \ref{th:bo}, the invertibility follows  if the $\phi_t(\cdot,\theta_0)$ are Lipschitz maps such that for some $x\in E$ and $r>0$, 
\begin{eqnarray}\nonumber
\E[\log^+(d(x,\phi_0(x,\theta_0)))]<\infty ,&\hspace{-1cm}\E[\log^+ \Lambda(\phi_0(\cdot,\theta_0) )]<\infty\\
 \label{eq:inv}&\mbox{and}\quad\E[\log \Lambda(\phi_0(\cdot,\theta_0)^{(r)})]<0.
\end{eqnarray}
The Remark \ref{rem:dim} also holds for the SRE driven by $(\phi_t)$: the metric space $(E,d)$ must be chosen carefully. The conditions \eqref{eq:inv} (with the optimal metric space $(E,d)$) are called the conditions of invertibility. 
\begin{Prop}
Under {\bf (ST)} and \eqref{eq:inv}, the general model \eqref{SREinftrue} is invertible.
\end{Prop}
Another notion of invertibility is introduced in \citet{Straumann2005}. We call it observable invertibility. As $\phi_t(x,\theta)=\phi_{\theta,x}(X_t,X_{t-1},\ldots)$ where $\phi_{\theta,x}$ is measurable for any $x,\theta$, denote $\hat \phi_t (x,\theta)=\phi_{\theta,x}(X_t,\ldots,X_1,u)$ where $u$ is an  arbitrary deterministic  sequence of $E^\N$.
\begin{Def}\label{def:oinv}
The model is observably invertible if and only if  the  SRE
\begin{equation}\label{SREtrue}
(h(\hat\Sigma_k))_{k\le t}=\hat\phi_t((h(\hat\Sigma_k))_{k\le t-1},\theta_0) \qquad t\ge 1
\end{equation}
is convergent for any arbitrary initial values $h(\hat\Sigma_k))_{k\le 0}$ and such that $\|\hat\Sigma_t-\Sigma_t\|\to 0$ in probability as $t\to\infty$.
\end{Def}
Notice that in general the approximative SRE does not fit the conditions of Theorem \ref{th:bo} and  in particular $(\hat\phi_t)$ is not necessarily stationary and ergodic. However, the Proposition below gives sufficient conditions for observable invertibility. It is a very useful result for the sequel of the paper, see Remark \ref{rem:dim2}. Notice that  logarithmic moments  are assumed and Theorem \ref{prop:log} is very useful to check this condition.
\begin{Prop}\label{prop:obinv}
If {\bf (ST)} and \eqref{eq:inv} hold, if the link function $\ell$ is continuous and it exists $x\in E$ such that $d(\hat\phi_t(x),\phi_t(x))\xrightarrow{\mbox{e.a.s.}} 0$ and $\Lambda(\hat\phi_t(\cdot,\theta_0)-\phi_t(\cdot,\theta_0))\xrightarrow{\mbox{e.a.s.}} 0$ as $t\to\infty$,
then the model is observably invertible.
\end{Prop}
\begin{proof} 
One can extend the proof of Theorem 2.10 in \citet{SM2006} written for Banach spaces to the case of the complete separable metric space $(E,d)$. That $d((h(\Sigma_k))_{k\le t}),(h(\hat\Sigma_k))_{k\le t}))\xrightarrow{\mbox{e.a.s.}} 0$ follows from the proof of Theorem 2.10 in \citet{SM2006} under the following assumptions:
\begin{description}
\item[S1] $\E[\log^+(d(x,\phi_0(x,\theta_0)))]<\infty$ for some $x\in E$,
\item[S2]  $\E[\log \Lambda(\phi_0(\cdot,\theta_0) )]<\infty$ and $\E[\log \Lambda(\phi_0(\cdot,\theta_0)^{(r)})]<0$ for some $r>0$,
\item[S2'] $\E[\log^+(d(y,(h(\Sigma_k))_{k\le t}))]<\infty$ for all $y\in E$,
\item[S3] $d(\hat\phi_t(x),\phi_t(x))\xrightarrow{\mbox{e.a.s.}} 0$ and $\Lambda(\hat\phi_t(\cdot,\theta_0)-\phi_t(\cdot,\theta_0))\xrightarrow{\mbox{e.a.s.}} 0$ as $t\to\infty$.
\end{description}
The conditions {\bf S1-S2} are equivalent to the invertibility conditions \eqref{eq:inv}. {\bf S3} holds from the assumptions in Proposition \ref{prop:obinv} and {\bf S2'} follows from {\bf (ST)}. Finally, using the continuity of the projection on the first coordinate and the one of the link function $\ell$, the desired result follows.
\end{proof}

\begin{rem}\label{rem:dim2}
Classical models such that GARCH(p,q) or EGARCH(p,q) models satisfy an  SRE for finite $p$ lags  $(h(\Sigma_k))_{t-p+1\le k\le t}=\phi_t((h(\Sigma_k))_{t-p\le k\le t-1},\theta_0)$ and for some $\phi_t$ generated by only a finite of past observation $(X_{t-1},\ldots,X_{t-q})$. In this context, the approximative SRE   coincides with the initial ones, i.e. one can choose $\hat\phi_t=\phi_t$ for $t> q$. Therefore, conditions of Proposition \ref{prop:obinv} hold systematically; invertibility and observable invertibility are equivalent, i.e. they are induced by the same Lyapunov condition. As for any initial values of $\hat\phi_t$ (for $0\le t\le q$)   the conditions of Proposition \ref{prop:obinv} are satisfied, seeking simplicity we  work in the sequel with $\hat\phi_t=\phi_t$ for all $t\ge 1$.
\end{rem}

\subsection{The continuous invertibility}

Now we are interested in inferring the unknown parameter $\theta_0$. To that ends, we extend the notions of invertibility at the point $\theta_0$ to the compact set $\Theta$ used in the definition \eqref{eq:QMLE_hat_theta} of our estimator. For the sake of simplicity, we assume in all the sequel that we are in the framework of Remark \ref{rem:dim2}, i.e. $\phi_t$ is observable for $t>q$ and $\hat \phi_t=\phi_t$. For any $\theta\in\Theta$, 
let us consider from  the functional SRE of the form
\begin{equation}\label{SREgen}
(\hat{g}_{k}(\theta ))_{t-p+1\le k\le t}= \phi_t((\hat{g}_{k}(\theta ))_{t-p\le k\le t-1},\theta),\qquad\forall t\ge 1,
\end{equation}
with arbitrary initial values $(\hat{g}_{k}(\theta ))_{1-p\le k\le 0}$. If the model is invertible for all $\theta\in\Theta$, the function $\hat g_t(\cdot)$ is well defined and converges to $g_t(\cdot)$, the unique stationary solution of \eqref{SREgen}. For statistical inference, the regularity of $g_t(\cdot)$ is required. We call uniform integrability the notion used in \cite{Straumann2005} and \citet{SM2006}:
\begin{Def}
The model is uniformly invertible on $\Theta$ if and only if the SRE \eqref{SREgen} is convergent when considering it on the Banach space of continuous functions $C(\Theta)$.
\end{Def}
This notion is too restrictive and we introduce the weaker notion of continuous invertibility as follows
\begin{Def}
The model is continuously invertible on $\Theta$ if and only if the SRE \eqref{SREgen} is convergent for all $\theta\in\Theta$ and the stationary solution $g_t(\cdot)$ is continuous.
\end{Def}
We have seen that sufficient conditions for pointwise invertibility are classically expressed in term of Lyapunov conditions. Let us consider models with parametric functions having continuous Lipschitz 
coefficients:

\begin{description}
\item[(CL)] For any metric spaces $\mathcal X$, $\mathcal Y$ and $\mathcal Z$, a function 
$f: \mathcal X\times \mathcal Y\mapsto \mathcal Z$ satisfies ${\bf (CL)}$ if there exists a continuous function 
$\Lambda_f: \mathcal Y\mapsto \R^+$ such that $\Lambda( f(\cdot,y))\le \Lambda_f(y)$ for all $y\in\mathcal Y$.
\end{description}
In this context, the uniform invertibility holds under the Lyapunov condition $\E[\log \sup_\Theta \Lambda_{\phi_0}^{(r)}(\theta)]<0$. This condition is too restrictive to handle EGARCH type models because of the supremum inside the expectation. We introduce a sufficient  condition {\bf (CI)} for continuous invertibility in term of a weaker Lyapounov condition than the preceding one:
\begin{description}
\item[(CI)] Assume that the SRE \eqref{SREgen} holds with  $\phi_t$ satisfying 
{\bf (CL)} for stationary $(\Lambda_{\phi_t})$ such that  conditions there exists an positve integr  $r$ such that $\E[\log \Lambda_{\phi_0}^{(r)}(\theta)]<0$  on the compact set $\Theta$. Assume moreover that $\E[\sup_\Theta \log^+  \Lambda_{\phi_0}^{(r)}(\theta)]<\infty$  and that there exists $y\in E$ such that $\E[\sup_\Theta\log^+(d(\phi_0(y,\theta),y))]<\infty$.
\end{description}
The condition {\bf (CI)} \label{con:CI}implies the convergence of the SRE \eqref{SREgen} for all $\theta\in {\Theta}$. If $\theta_0\in\Theta$ it implies the invertibility of the model as described in Subsection \ref{subsec:inv}. It also implies the local uniform regularity of $g_t(\cdot)$  and thus the continuous invertibility:

\begin{Theo}\label{contsol}\label{th:cont}
Assume that {\bf (ST)} and {\bf (CI)} hold. Then the functions 
$g_t(\cdot)$ are continuous for all  $\theta\in {\Theta}$ and all $t\in\Z$. Moreover, for any
 $\theta\in {\Theta}$ 
there exists an $\epsilon>0$ such that $\hat{g}_{t}(\theta)$ satisfying
\eqref{SREgen} satisfies 
\begin{equation}\label{expoapprox}
\lim\sup_{\theta'\in \overline B(\theta,\epsilon)\cap \Theta}d(\hat{g}_{t}(\theta'),g_{t}(\theta'))\xrightarrow{\mbox{e.a.s.}}0.
\end{equation}
\end{Theo}

 \begin{proof} Without loss of generality, one can assume that $\E[\log\Lambda_{\phi_0}^{(r)}(\theta)]>-\infty$ such that $\lim_{K\to \infty}\E[\log \Lambda_{\phi_0}^{(r)}(\theta)\vee (-K)] =\E[\log \Lambda_{\phi_0}^{(r)}(\theta)]$ for any $\theta\in\Theta$. For any $\rho>0$, let us write
 $\Lambda_\ast^{(r)}(\theta,\rho)=\sup\{\Lambda_{\phi_0}^{(r)}(\theta'),\theta'\in \overline B(\theta,\rho)\cap \Theta\}$, where $\overline B(\theta,\rho)$ stands 
for the closed ball centered at $\theta$ with radius $\rho$. Noticing that $\E[\sup_\Theta|\log \Lambda_{\phi_0}^{(r)}(\theta)\vee (-K)|]<\infty$, 
using the dominated convergence Theorem we obtain
$\lim_{\rho\to0}\E(\log \Lambda_{\ast}^{(r)}(\theta,\rho)\vee (-K))=\E(\lim_{\rho\to0}\log \Lambda_{\ast}^{(r)}(\theta,\rho)\vee (-K))$.
By continuity $\lim_{\rho\to0}\log \Lambda_{\ast}^{(r)}(\theta,\rho)= \log \Lambda_{\phi_0}^{(r)}(\theta)$ and for sufficiently large $K$
$$
\lim_{\rho\to0}\E(\log \Lambda_{\ast}^{(r)}(\theta,\rho)\vee (-K))= \E(\log \Lambda_{\phi_0}^{(r)}(\theta)\vee (-K))<0.
$$
Thus, there exists an $\epsilon>0$ such that 
$$\E(\log \Lambda_{\ast}^{(r)}(\theta,\epsilon))\le\E(\log \Lambda_{\ast}^{(r)}(\theta,\epsilon)\vee (-K))<0.$$

Let us now work on $C(\overline B(\theta,\epsilon)\cap \Theta)$, the complete metric space of continuous functions from $\overline B(\theta,\epsilon)\cap\Theta$ to $\R$ equipped with the supremum norm $d_\infty=\sup_{\overline B(\theta,\epsilon)\cap\Theta} d$. In this setting $(\hat g_t)$ satisfy a functional SRE $(\hat g_k)_{k\le t} = \tilde\phi_t((\hat g_{k})_{k\le t-1})$ with Lipschitz constants satisfying
\begin{align*}
\Lambda_\infty(\tilde\phi_t^{(r)}(\cdot))&\le \sup_{s_1,s_2\in C(\overline B(\theta,\epsilon)\cap\Theta)}\frac{d_\infty(\tilde\phi_t^{(r)}(s_1),\tilde\phi_t^{(r)}(s_2))}{d_\infty(s_1,s_2)}\\
&\le \sup_{s_1,s_2\in C(\overline B(\theta,\epsilon)\cap\Theta)}\frac{\sup_{\overline B(\theta,\epsilon)\cap\Theta}d(\phi_t^{(r)}(s_1(\theta'),\theta'),\phi_t^{(r)}(s_2(\theta'),\theta')}{d_\infty(s_1,s_2)}\\
&\le \sup_{s_1,s_2\in C(\overline B(\theta,\epsilon)\cap\Theta)}\frac{\sup_{\overline B(\theta,\epsilon)\cap\Theta}\Lambda(\phi_t^{(r)}(\cdot,\theta'))d(s_1(\theta')s_2(\theta'))}{d_\infty(s_1,s_2)}\\
&\le  \sup_{s_1,s_2\in C(\overline B(\theta,\epsilon)\cap\Theta)}\frac{\sup_{\overline B(\theta,\epsilon)\cap\Theta}\Lambda(\phi_t^{(r)}(\cdot,\theta'))d_\infty(s_1,s_2)}{d_\infty(s_1,s_2)}\\
&\le \Lambda_{\ast}^{(r)}(\theta,\epsilon).
\end{align*}
As $\E[\sup_{\overline B(\theta,\epsilon)\cap\Theta}\log^+(d(\phi_t(y,\theta'),y))]\le \E[\sup_\Theta\log^+(d(\phi_t(y,\theta),y))] $ is finite and $ \E(\log \Lambda_{\phi_0}^{(r)}(\theta))<0$ we apply Theorem \ref{thm:Sol of SRE}. 
 By recurrence $\phi_{t}\circ\cdots\circ\phi_{t-m}(\zeta_{0})\in C(\overline B(\theta,\epsilon)\cap\Theta)$ is continuous 
in $\theta$  and  so is $g_t$ as the convergence holds uniformly on $\overline B(\theta,\epsilon)\cap\Theta$. 
It is true for any $\theta\in {\Theta}$ and the result follows.
\end{proof}

\section{Statistical inference under continuous invertibility}\label{sec:inf}

Consider  $\hat{\theta}_{n}=\mbox{argmin}_{\theta\in\Theta}\hat{S}_{n}(\theta)$
the M-estimator associated with the QLIK criteria \eqref{eq:hat_L} where $(\hat g_t)$ is obtained from the approximative SRE \eqref{SREgen}. 
\begin{rem}\label{rem:qmle} The statistical procedure described here does not coincide  with the Quasi Maximum Likelihood for non uniformly invertible models. For a detailed discussion in the EGARCH(1,1) case see \cite{WC}.
\end{rem}
\subsection{Strong consistency of the parametric inference}
From now on, we assume that the innovations process $(Z_t)$ is iid: 
\begin{description}
\item[(IN)]  The $Z_t$ are iid variables such that $\E[Z_0^TZ_0] $ is the identity matrix.
\end{description}
The next assumption 
implies that the volatility matrices are  invertible and that the link function $\ell$ is continuous:
\begin{description}
\item[(IV)]  The functions $\ell^{-1}$ and $\log(\det(\ell))$ are Lipschitz 
satisfying 
$\det(\ell(g_0(\theta)))\ge C(\theta)$ for some continuous function $C:\Theta\mapsto (0,\infty)$.
\end{description}
\begin{rem}\label{rem:lf}
The SRE criteria converges to the  possibly degenerate limit
$$
S(\theta)=\E[s_0(\theta)]= 2^{-1}\E\left[ X_{0}^T  \ell( g_{0}(\theta)) ^{-1}X_{0}
+\log(\det[\ell( g_{0}(\theta))])\right]
$$
Notice that $S(\theta_0)=2^{-1}\E[Z_0^TZ_0+\log(\det(\Sigma_0))]$ is finite under {\bf (ST)}, {\bf (IN)} and {\bf (IV)} because $h(\Sigma_0)$ has logarithmic moments and  $\log(\det(\ell))$ is Lipschitz. It is a considerable advantage of the QLIK criteria: it does not need moments of any order to be defined in $\theta_0$. Even if $S(\theta)$ may be equal to $+\infty$ for $\theta\neq\theta_0$, it does not disturb the statistical procedure that defines $\hat\theta_n$ as a minimizer.
\end{rem}
If the model is identifiable, the estimator $\hat\theta_n$ is strongly consistent:
\begin{Theo}\label{th:sc}
Assume that {(\bf ST)} and  {\bf (CI)} are satisfied on the compact set $\Theta$. If {\bf (IN)} and {\bf (IV)} are satisfied and the model is identifiable, i.e. $g_0(\theta)=h(\Sigma_0 )$
 iff $\theta=\theta_0$, then $\hat\theta_n\to \theta_0$ a.s. for any $\theta_0\in {\Theta}$.
\end{Theo}

\begin{proof}With no loss of generality we  restrict  our propose 
to $\theta\in\Theta$ satisfying the relation $\det(\ell(\hat \phi_t(\cdot,\theta)))\ge C(\theta)$ for all $t\le 0$.  We adapt the proof of  \citet{Jeantheau1993} and its notation 
$  s_{\ast t}(\theta,\rho)=\inf\{ s_t(\theta'), \theta'\in \overline B(\theta,\rho)\}$ and
 $\hat s_{\ast t}(\theta,\rho)=\inf\{\hat s_t(\theta'), \theta'\in \overline B(\theta,\rho)\}$.
Let us recall Theorem 5.1 in \citet{Jeantheau1993} : The M-estimator associated with the loss \eqref{eq:hat_L}
is strongly consistent under the hypothesis {\bf H1-H6}:
\begin{description}
\item[H1]\label{ite:H1}  $\Theta$ is
compact. 
\item[H2]\label{ite:H2} $\hat S_n(\theta)\to S(\theta)$ a.s. under the stationary law $P_{\theta_{0}}$.
\item[H3]\label{ite:H3} $S(\theta)$ admits a unique minimum for $\theta=\theta_0$ in $ {\Theta}$. 
Moreover for any $\theta_1\neq\theta_0$ we have:
$$
\lim\inf_{\theta\to\theta_1} S(\theta)>S(\theta_1).
$$
\item[H4]\label{ite:H4} $\forall \theta\in\Theta$ and sufficiently small $\rho>0$ the process 
$(\hat s_{\ast t}(\theta,\rho)))_t$ is ergodic. 
\item[H5] \label{ite:H5} $\forall \theta\in\Theta$, $\E_{\theta_0}[s_{\ast1}(\theta,\rho)]>-\infty$.
\item[H6] \label{ite:H6} $\lim_{\rho\to0}\E_{\theta_0}[s_{\ast1}(\theta,\rho)]=\E[s_{\ast1}(\theta)]$.
\end{description}
Let us check {\bf H1-H6} in our case. {\bf H1} is satisfied by assumption. {\bf H2} is verified in two steps. First, by the e.a.s.
 convergence given by Theorem \ref{contsol}, arguments of \citet{Straumann2005} and the Lipschitz properties of $\ell^{-1}$ and $\log(\det(\ell))$ we obtain
$$
\frac 1n \sum_{t=1}^n\sup_{\overline B(\theta,\epsilon)}|\hat s_t(\theta')-s_t(\theta')|\to 0 \qquad P_{\theta_0}-a.s.
$$
Second we use that $(s_t)$ is an ergodic sequence and Proposition 1.1 of \citet{Jeantheau1993}. As the $s_t$ are bounded from below $n^{-1}\sum_{t=1}^n s_t(\theta)$
 converges $P_{\theta_0}$-a.s. to $S(\theta)$ (taking values in $\R\cup\{+\infty\}$)  
$$
\frac 1n \sum_{t=1}^n | s_t(\theta)-S(\theta)|\to 0 \qquad P_{\theta_0}-a.s.
$$
Combining this two steps leads to {\bf H2}. The first part of {\bf H3} is checked similarly than in (ii) p.2474 of \citet{SM2006} and with the help of the Remark \ref{rem:lf}. Notice that $S$ has a unique minimum iff $\E [\mbox{Tr}(\Sigma_0\cdot\ell(g_0(\theta))^{-1})-\log(\det(\Sigma_0\cdot\ell(g_0(\theta))^{-1} ))] $ has a unique minimum. As this criteria is the integrand of a sum of the $\lambda_i-\log(\lambda_i)$ where the $\lambda_i$ are positive eigenvalues, we conclude under the identifiability condition from the property $x-\log(x)\ge 1$ for all $x>0$ with equality iff $x=1$. 
The second part is checked using the fact that 
$$
\lim \inf_{\theta\to\theta_1}S(\theta)\ge \E[\lim \inf_{\theta\to\theta_1}s_0(\theta)]=\E[ s_0(\theta_1)]=S(\theta_1)
$$ where the first inequality was already used for proving Theorem \ref{contsol} and the first equality comes from the 
local continuity of $g_0$ and $\ell$. {\bf H4} is satisfied from the ergodicity of $(\hat s_t)$. {\bf H5} and {\bf H6} follows from Theorem \ref{contsol}
that ensures the continuity of the function $s_{\ast 1}$ and by the lower bounded assumption on $\det( \ell)$, 
see Proposition 1.3 of \citet{Jeantheau1993}. 
\end{proof}

\subsection{Volatility forecasting}

From the inference of $\theta_0$, we deduce a natural forecast of the volatility $\hat\Sigma_t=\ell(\hat g_t(\hat\theta_t))$. It is strongly consistent:

\begin{Theo}\label{pr:vf}
Under the conditions of Theorem \ref{th:sc} then  $\|\hat\Sigma_t- \Sigma_t\|\to 0$ a.s. as $t\to\infty$.
\end{Theo}
\begin{proof}
It is a consequence of Theorems \ref{th:cont} and \ref{th:sc} that assert the a.s. convergence of
 $\hat\theta_t$ toward $\theta_0$ and the local uniform convergence of $\hat g_t$ toward $g_t$. Notice that for $t$ sufficiently large such that $\hat \theta_t\in\overline B(\theta,\epsilon)$, a ball where the uniform Lyapunov condition  
$\E[ \log \Lambda_\infty(\phi_t(\cdot))]<0$ is satisfied. Thus $\hat g_t(\hat \theta_t)- g_t(\theta_t) \to 0$ a.s. and by continuity of $\ell$ and $g_t$ 
and from the identification $\Sigma_t=\ell( g_t( \theta_0))$ the result follows if $g_t(\hat \theta_t)$ converges to $g_t( \theta_0)$.
 For proving it,  we use
$$
d(g_t(\hat \theta_t),g_t( \theta_0))\le  \Lambda_\infty(\phi_t(\cdot))d(g_{t-1}(\hat \theta_t),g_{t-1}( \theta_0))+w_t(\hat\theta_t)$$
where $w_t(\hat\theta_t)=d(\phi_t(g_{t-1}(\theta_0),\hat \theta_t),\phi_t(g_{t-1}(\theta_0), \theta_0))$. The quantity $d(g_t(\hat \theta_t),g_t( \theta_0)) $ is upper bounded by $v_t$ satisfying the SRE of linear stationary maps $v_t= \Lambda_\infty(\phi_t(\cdot))v_{t-1}+w_t(\hat\theta_t)$. We apply Theorem \ref{th:bo}
as $\E[\log(\Lambda_\infty(\phi_t(\cdot)))]<0$ and $\E\log^+( w_0(\hat\theta_t ))\le \E[\sup_\Theta\log^+(2d(\phi_t(y,\theta),y))]<\infty$. We get
 $$
d(g_t(\hat \theta_t),g_t( \theta_0))\le \sum_{i=0}^\infty\Lambda_\infty(\phi_t(\cdot))\cdots\Lambda_\infty(\phi_{t-i+1}(\cdot))w_{t-i}(\hat\theta_t).
$$
Conditioning on $(\hat\theta_t)$, the  upper bound is a stationary normally convergent series of functions and
\begin{align*}
&\P\Big(\sum_{i=0}^\infty\Lambda_\infty(\phi_t(\cdot))\cdots\Lambda_\infty(\phi_{t-i+1}(\cdot))w_{t-i}(\hat\theta_t)\to 0 \Big)\\
&=\E\Big[\P\Big(\sum_{i=0}^\infty\Lambda_\infty(\phi_t(\cdot))\cdots\Lambda_\infty(\phi_{t-i+1}(\cdot))w_{t-i}(\hat\theta_t)\to 0 ~|~(\hat\theta_t)\Big)\Big]\\
&=\E\Big[\P\Big(\sum_{i=0}^\infty\Lambda_\infty(\phi_0(\cdot))\cdots\Lambda_\infty(\phi_{-i+1}(\cdot))w_{-i}(\hat\theta_t)\to 0 ~|~(\hat\theta_t)\Big)\Big] \\&= \E[1]=1,
\end{align*}
the last inequalities following from the continuity of normally convergent series of functions, $\hat\theta_t\to\theta_0$ and $w_i(\hat\theta_t)\to w_i( \theta_0)=0$ for all $i$ a.s. as $t\to \infty$.
 \end{proof}

\subsection{Asymptotic normality  of the parametric inference}

Classical computations show  that if the M-estimator $\hat{\theta}_{n}$ is asymptotically normal then 
the asymptotic variance is given by the expression 
$$\bfV=\bfP^{-1}\bfQ\bfP^{-1}$$
with $\bfP=\E[\mathbb H s_0(\theta_0)]$
and $\bfQ=\E[\nabla s_0(\theta_0)\nabla s_0(\theta_0)^T]$, where $\mathbb H s_0(\theta_0)$ and $\nabla s_0$ are the Hessian matrix and the gradient vector of $s_0(\theta_0)$.    \begin{description}
\item[(AV)] Assume that $\E(\|Z_0Z_0^T\|^2)<\infty$ and that the functions $\ell$ and $\phi_t$ 
are 2-times continuously differentiable on the compact set $\Theta$ that coincides with the closure of its interior. 
\end{description}
The following moments assumptions ensure the existence of $\bfQ$ and $\bf P$:
\begin{description}
\item[(MM)] Assume that 
$\E[ \|\nabla s_0(\theta_0) \|^2]<\infty$ and  $\E[\| \mathbb H s_0(\theta_0)\|]<\infty$.
\end{description}
These moments assumptions holds only for $\theta=\theta_0$; they are simpler to verify than for the moment conditions for $\theta\neq\theta_0$ due to the specific form of the derivatives of  the SRE criteria at $\theta_0$, see  Remark \ref{rem:lf} for details.
The next assumption is classical and ensures to the existence of  $\bfP^{-1}$:
\begin{description}
\item[(LI)] The components of the vector $\nabla g_0(\theta_0)$ are linearly independent.
\end{description}
Let $\mathcal V=\overline B(\theta_0,\epsilon)\subset \Theta$  with $\theta_0\in \stackrel{\circ}\Theta$ and $\epsilon>0$ chosen in accordance with Theorem \ref{th:cont}, i.e. such that $\E[\log(\sup_{\mathcal V}\Lambda_{\phi_0})]<0$. The two next assumptions are specific to the SRE approach. They ensure  that $\nabla \hat s_t(\theta)$ is a good approximation of $\nabla   s_t(\theta)$ uniformly on the neighborhood $\mathcal V$ of $\theta_0$:
\begin{description}
\item[(DL)]  The partial derivatives $\Phi_t=D_x(\phi_t)$, $=D_\theta(\phi_t)$, $=D_{x^2}^2(\phi_0) $,
 $D_{\theta,x}^2(\phi_0) $ or $D_{\theta^2}^2(\phi_0)$ satisfy {\bf (CL)} for stationary 
$(\Lambda_{\Phi_t})$ with $\E[\sup_{\mathcal V}\log (\Lambda_{\Phi_0} )]<\infty$. Assume there exists $y\in E$ such that $\E[\sup_{\mathcal V}(\log^+(d(\phi_0(y,\theta),y)))^2]<\infty$.
\item[(LM)] Assume that $y\to\nabla \ell^{-1}(y)$ and $y\to \nabla \log(\det(\ell(y)))$ are Lipschitz functions.
\end{description}
Now that $\bfV$ is well defined in terms of derivatives of $s_0$ that are well approximated by derivatives of the SRE, the procedure is asymptotically normal:

\begin{Theo}\label{th:an}
Under the assumptions of Theorem \ref{th:sc}, {\bf (AV)}, {\bf (MM)}, {\bf (LI)},  {\bf (DL)} and {\bf (LM)} then the asymptotic variance $\bfV$ is a well defined invertible matrix and $$\sqrt n(\hat\theta_n-\hat\theta_0)\stackrel{d}{\to} \mathcal N(0,\bfV)\quad\mbox{if}\quad \theta_0\in\stackrel{\circ}{\Theta}.$$ 
\end{Theo}
\begin{proof} Under {\bf (CI)} and {\bf(LM)} we apply Theorem \ref{prop:log} to $(\sup_{\mathcal V}g_t(\theta))$. We obtain $\E[ \log^+(\sup_{\mathcal V}g_0(\theta))]<\infty$. From the Lipschitz condition in {\bf (DL)} we have that $\E[\sup_{\mathcal V}\log^+(\|\Phi_0(\theta)\|)]<\infty$ for $\Phi_0(\theta)=D_\theta(\phi_0)(g_0(\theta),\theta)$ or $ D_{x^2}^2(\phi_0)(g_0(\theta),\theta)$ or
 $D_{\theta,x}^2(\phi_0)(g_0(\theta),\theta)$ or $D_{\theta^2}^2(\phi_0)(g_0(\theta),\theta)$. Using the existence of these logarithmic moments and the relation $\E[\log(\sup_{\mathcal V}\Lambda_{\phi_0})]<0$, we apply recursively the Theorem \ref{th:bo} and prove the existence 
of continuous first and second derivatives of $(g_t(\theta))$ on $\mathcal V$ as solutions of functional SRE. 
The asymptotic normality follows from the Taylor development of Section 5 of \citet{BW2009} on the first partial derivatives $\nabla_i$ of 	$S_n$
$$
\nabla_iS_n(\hat\theta_n)-\nabla_iS_n( \theta_0)=\mathbb H S_n(\tilde\theta_{n,i})(\hat\theta_n-\theta_0)\quad\mbox{for all}~1\le i\le d.
$$
Then the asymptotic normality follows from the following sufficient conditions:
\begin{enumerate}
\item  $n^{-1/2}\nabla S_n(\theta_0)\to \mathcal N(0,\bfQ)$,
\item  $\|n^{-1}\mathbb H S_n(\tilde\theta_n)-\bfP\|$ converges a.s. 
to $0$ for any sequence $(\tilde \theta_n)$ converging a.s. to $\theta_0$ and  $ \bfP$  is invertible,
\item  $ n^{-1/2}\|\nabla \hat S_n(\hat\theta_n)-\nabla S_n(\hat\theta_n)\|$
  converges a.s. to $0$.
\end{enumerate}
Due to its specific expression and  Assumption {\bf (IN)} the process $(\nabla S_n(\theta_0))$ is a martingale. Under {\bf (MM)}, the CLT for  martingales applied to $(\nabla S_n(\theta_0))$ leads to the first condition. The first part of the second condition is derived
 from similar arguments than in the proof of Theorem \ref{pr:vf} and an application of the Cesaro mean theorem ensuring that
$
n^{-1} \|\mathbb H S_n(\tilde\theta_n)-\sum_{t=1}^n\mathbb H s_t(\theta_0) \|\to 0
$ a.s.
The ergodic Theorem on $(\mathbb H s_t(\theta_0))$ with {\bf (MM)} leads to $\|n^{-1}\mathbb H S_n(\tilde\theta_n)-\bfP\|\to 0$ a.s.
The fact that $\bfP$ is invertible follows from {\bf (LI)}, see \citet{BW2009} for detailed computations. Finally the third condition is obtained 
by applying Theorem 2.10 of \cite{SM2006} to the SRE satisfied by $(\nabla g_t )$ and its approximative SRE satisfied by $(\nabla \hat g_t )$ uniformly on $\mathcal V$. 
Thus $ \sup_{\mathcal V}\|\nabla \hat g_t-\nabla g_t \|\xrightarrow{\mbox{e.a.s.}}0$ as $t\to\infty$ and Lipschitz conditions on $\nabla \ell^{-1}$ and $\nabla \log(\det(\ell))$ in {\bf (LM)} and arguments similar than in \citet{Straumann2005} leads to the desired result. \end{proof}

\section{Applications to GARCH and EGARCH type models}\label{sec:egarch}

\subsection{Some applications to GARCH type models}\label{subs:garch}
In  GARCH type models, the stationarity assumption {\bf (ST)} is crucial, whereas the continuous invertibility condition {\bf (CI)} is automatically satisfied on well chosen compact sets $\Theta$ due to the form of the model (they are also uniformly invertible). The asymptotic properties of the QMLE (that coincides with the procedure described here in these cases, see Remark \ref{rem:qmle}) follow from Theorems \ref{th:sc} and \ref{th:an}. Thus, we recover and slightly refine existing results in the AGARCH and CCC-GARCH models (we refer the reader to \citet{Straumann2005} and \citet{Francq2011} respectively for details). Moreover, we prove also for the first time the strong consistency of the natural volatility forecast in both cases under the conditions of strong consistency of $\hat \theta_n$.\\

First, let us detail the case of the univariate  APGARCH($p,q$) model introduced in  \citet{Ding1993}, \citet{Zakoian1994} and studied in \citet{Straumann2005}:
$$
\sigma_t^2 = \alpha_0+\sum_{i=1}^p\alpha_i(|X_{t-i}|-\gamma X_{t-i})^2+\sum_{j=1}^q\beta_j\sigma_{t-j}^2, \qquad t\in\Z,
$$
where $\alpha_0>0$, $\alpha_i,$ $\beta_j\ge 0$ and $|\gamma|\le 1$ (it coincides  with the GARCH($p,q$) model if $\gamma=0$. Then we derive the strong consistency and the asymptotic normality directly from our Theorems \ref{th:sc} and \ref{th:an}. The conditions we obtained coincides with these of Theorem 5.5 and Theorem 8.1 of \citet{SM2006} except their useless condition (8.1)   as one does not need moment of any order.\\

Second, let us consider the multivariate CCC-GARCH($p,q$) model introduced by \citet{Bollerslev1990}, first studied in \citet{Jeantheau1998} and refined in \citet{Francq2011}
$$
Diag(\Sigma_t^2)= A_0+ \sum_{i=1}^q A_iDiag(X_{t-i}X_{t-1}^T)+\sum_{i=1}^p B_i Diag(\Sigma_{t-i}^2)
$$
and $(\Sigma_t^2)_{i,j}=\rho_{i,j}\sqrt{(\Sigma_t^2)_{i,i}(\Sigma_t^2)_{j,j})}$ for all $(i,j)$, where $Diag(M)$ is the vector of the diagonal elements of $M$. A necessary and sufficient conditions for {\bf (ST)} is given in term of top Lyapunov condition in \citet{Francq2011}. We recover the strong consistency and the asymptotic normality of \citet{Francq2011} directly from our Theorems \ref{th:sc} and \ref{th:an}.

\subsection{Application to the EGARCH($1,1$) model}

Let $(Z_t)$ be an iid sequence of random variables not concentrated on two points such that  $\E(Z_0^2)=1$.
The EGARCH($1,1$) model introduced by \citet{Nelson1991} is an AR($1$) model for $\log \sigma_t^2$,
$$
X_{t}  =  \sigma_{t}Z_{t}\quad\mbox{with}\quad
\log\sigma_{t}^{2}  =\alpha_0+\beta_0\log\sigma_{t-1}^{2}+W_{t-1}(\theta_0)
$$
where $W_{t}(\theta_0)=\gamma_0 Z_{t}+\delta_0\left|Z_{t}\right|$ are the innovations of this AR($1$) model. 
Let $\theta_0=(\alpha_0,\beta_0,\gamma_0,\delta_0)$ be the unknown parameter. Assume that  $ |\beta_0|<1$ such that   there exists a stationary solution having a MA($\infty$) representation: 
\begin{equation}
\label{eq:Log_var_MA}
\log\sigma_{t}^{2}=\alpha_0(1-\beta_0)^{-1}+\sum_{k=1}^{\infty}\beta_0^{k-1}W_{t-k}(\theta_0).
\end{equation}
The moments assumptions on $Z_t$ ensures that the process $(\log \sigma_t^2)$ is ergodic, strongly and weakly stationary. 
Then the volatilities process $( \sigma_t^2)$ is also ergodic and strongly stationary and {\bf (ST)} holds. However, it does not necessarily 
have finite moment of any order.\\

The invertibility of the stationary solution of the EGARCH($1,1$) model does not hold in general. A sufficient condition for invertibility is given in \citet{SM2006}.  As $(\log\sigma_{t}^{2})$ satisfies the SRE 
$$
\log \sigma_t^2=\alpha_0+\beta_0 \log \sigma_{t-1}^2+(\delta_0 X_{t-1}+ \gamma_0 |X_{t-1}|)\exp (- {\log \sigma_{t-1}^2}/{2} ),
$$
if it has a non anticipative solution the model is invertible. Keeping the notation of  Section 2, the function $h$ is now the logarithmic function and the SRE \eqref{SREtrue} holds with $(\phi_t)$ defined by
$$
\phi_{t}(\cdot;\theta):s  \mapsto  \alpha+\beta s+(\gamma X_{t-1}+\delta\left|X_{t-1}\right|)\exp(-s/2)\label{eq:Link function phi}
$$
We check that the $\phi_t$ are random functions generated by $\mathcal G_{t-1}$.
For any $\theta\in\R\times \R^+\times\{\gamma\ge |\delta|\}$ we restrict $\phi_{t}(\cdot;\theta)$
on the complete separable metric space $[\alpha/(1-\beta),\infty)$ equipped with $d(x,y)=|x-y|$. The process $(\phi_t(\cdot;\theta ))$ is a stationary ergodic sequence of Lipschitz maps from  
$[\alpha /(1-\beta ),\infty)$ to $[\alpha /(1-\beta ),\infty)$ with  Lipschitz coefficients
$$
\Lambda(\phi_{t}(\cdot,\theta_0))\le \max \{\beta ,2^{-1}(\gamma  X_{t-1}+\delta |X_{t-1}|)\exp(-2^{-1}\alpha /(1-\beta ))-\beta\}.
$$
The technical smoothness assumption  {\bf(CL)} is automatically satisfied as
$$
(\Lambda_{\phi_t}(\theta))_t=(\max(\beta,2^{-1}(\gamma X_{t-1}+\delta|X_{t-1}|)\exp(-2^{-1}\alpha/(1-\beta))-\beta))_t
$$ 
is a stationary process of continuous functions of $\theta$. The EGARCH(1,1) is continuously invertible on any compact  set $\Theta\subset\R\times \R^+\times\{\gamma\ge |\delta|\}$ such that 
\begin{equation}
\E[\log(\max\{ \beta , 2^{-1}(\gamma  X_{0}+\delta |X_{0}|)\exp(-2^{-1}\alpha /(1-\beta ))-\beta \})]<0.\label{eq:condtheta}
\end{equation}
This sufficient condition for continuous invertibility depends on the distribution of the observations $(X_t)$.
Notice that if $\theta_0\in\Theta$ then it satisfies the condition of stationarity $\beta_0<1$ and, from the MA($\infty$) representation  \eqref{eq:Log_var_MA} of $\log\sigma_t^2$, our condition \eqref{eq:condtheta} expressed at $\theta_0$ implies that
\begin{multline}
\E\Big[\log \Big(\max\Big \{\beta_0,2^{-1}\exp\Big(2^{-1}\sum_{k=0}^{\infty}\beta_0^{k}(\gamma_0 Z_{-k-1}+\delta_0 \left|Z_{-k-1}\right|)\Big)
\\\times(\gamma_0 Z_{0}+\delta_0 \left|Z_{0}\right|)-\beta_0\Big\}\Big)\Big]<0.
\label{eq:condtheta2}
\end{multline}
It is the condition of invertibility of the EGARCH(1,1) model given in \cite{SM2006}.  Applying our results of Theorem \ref{th:sc} and Theorem \ref{pr:vf}, we obtain
\begin{Theo}\label{cor:sc}
For any  compact subset $\Theta$ of $\R\times \R^+\times\{\gamma\ge |\delta|\}$ satisfying   \eqref{eq:condtheta} then $\hat\theta_n\to\theta_0$ and $\hat\sigma_n^2-\sigma_n^2\to 0$ a.s. as $n\to \infty$ with $\hat\sigma^2_t =\exp(\hat g_t(\hat \theta_n))$  if $\theta_0\in {\Theta}$.
\end{Theo}
\begin{proof}
The condition   
{\bf(CI)} follows from $\E[\log\Lambda(\phi_t (,\theta))]<0$ by assumption of $\Theta$ and $\E[\sup_{\Theta}\log\Lambda(\phi_t (,\theta))]<\infty$
 since $\E\log|X_{t-1}|=E(\log\sigma+\log|Z_{t-1}|)<\infty$ as $\log\sigma^2_t$ has a MA($\infty$) representation 
\eqref{eq:Log_var_MA} and $Z$ is integrable. Moreover as $\log^{+}(d(\phi_{0}(0,\theta),0))=\log^{+}|\alpha+(\gamma X_{-1}+\delta|X_{-1}|)|$ then  fixing $y=0$ one has
$\E[\sup_{\Theta}(\log^{+}(d(\phi_{0}(y,\theta),y))^2]<\infty$.\\
In the EGARCH($1,1$) model  the link function is the exponential function $\ell(x)=\exp(x)$ and since we have 
$\log\sigma^2_t\ge \alpha/(1-\beta)$, $1/\ell(x)=\exp(-x)$ is a Lipschitz function ($\log (\det(\ell))=id$ is also   a Lipschitz function).
Moreover the volatility process $(\sigma_t^2)$ is bounded from below by $C(\theta)=\exp(\alpha/(1-\beta))$. 
Finally, the identifiability condition  $ g_0(\theta)=h(\theta_0 )$ iff $\theta=\theta_0$ is checked in Section 5.1 of \citet{SM2006}. \end{proof}

Notice that the procedure is valid only if $\theta\in\Theta$ satisfies \eqref{eq:condtheta}. In practice, we suggest to optimize the QLIK criteria under the empirical constraint
$$
\sum_{t=1}^n\log(\beta,2^{-1}(\gamma X_{t}+\delta|X_t|)\exp(-2^{-1}\alpha/(1-\beta))-\beta)<0.
$$ 
Existing procedures does not constrain the model to be invertible. The corresponding empirical results are not asymptotically valid since the underlying SREs are not stable, see \citet{WC} for evidences on simulations. Notice also that the strong consistency of $\hat\theta_n$ has already been proved in \citet{Straumann2005} and \citet{SM2006} under the uniform invertibility condition
$$
\E[\sup_\Theta\log(\max\{ \beta , 2^{-1}(\gamma  X_{t-1}+\delta |X_{t-1}|)\exp(-2^{-1}\alpha /(1-\beta ))-\beta \})]<0.
$$
We share the appreciation of the authors  that this condition is very restrictive and might  not be satisfied for reasonable compact sets $\Theta$ with $\beta\neq0$ and non empty interior.\\

As a corollary of Theorem \ref{th:an} we get the asymptotic normality of the statistical inference in the EGARCH($1,1$) model. The only existing result was Theorem 5.7.9 of \citet{Straumann2005} valid only for degenerate EGARCH(1,1) models with $\beta=0$. Our result holds under consistency assumptions and the following necessary and sufficient condition for the existence of  $\bfV$:
\begin{description}
\item[(MM')]  $\E[Z_0^4]<\infty$ and $\E[(\beta_0 -2^{-1}(\gamma_0 Z_0+\delta_0\left|Z_0\right|)^2]<1$.
\end{description}
\begin{Theo}
Assume that assumptions of Theorem \ref{cor:sc} are satisfied and that {\bf (MM')} holds. Then   $\sqrt n(\hat\theta_n-\theta_0)\stackrel{d}{\to} \mathcal N(0,\bfV)$ where $\bfV$ is an invertible matrix  if $\theta_0\in \stackrel{\circ}{\Theta}$.

\end{Theo}

\begin{proof}
By definition, $(\phi_t)$ is 2-times continuously differentiable and simple computations give 
$D_x(\phi_t)(x,\theta)=\beta-2^{-1}(\gamma X_{t-1}+\delta |X_{t-1}|)\exp(-x/2),
 D_\theta(\phi_t)(x,\theta)=(1,x,X_{t-1}\exp(-x/2),|X_{t-1}|\exp(-x/2))^T,
D_{x^2}^2(\phi_t)(x,\theta)=4^{-1}(\gamma X_{t-1}+\delta  |X_{t-1}|)$ $\exp(-x/2),$ $D_{x,\theta}^2(\phi_t)(x,\theta)=(0,1,2^{-1}X_{t-1}\exp(-x/2),$ $2^{-1}|X_{t-1}|\exp(-x/2) )^T$
 and $D_{\theta^2}^2(\phi_t)(x,\theta)=0.$
 Moreover, as  the link function is $\ell(x)=\exp(x)$ is also 2-times continuously differentiable, the last assertion of the condition {\bf (AV)} of Theorem \ref{th:an}  holds. 
The fact that {\bf (MM)}  holds iff {\bf (MM')}  holds is technical and postponed to the Appendix.   The fact that {\bf (LI)}  holds if $Z_0$ is not concentrated on two points is classical, see  for instance Lemma 8.2 of \cite{SM2006}. Assumption {\bf (DL)} is satisfied from the expressions of the derivatives (that are Lipschitz functions) and as all the logarithmic moments are finite due to $\E[ \log (X_{t-1}^2)]<\infty$. Finally {\bf (LM)} is automatically satisfied due to the specific expression of the link function. Thus Theorem \ref{th:an} applies. \end{proof}

\subsubsection*{Acknowledgments} One of the author would like to thank C. Francq and J.-M. Zako\"ian for helpful discussions and for having pointing out some mistakes in an earlier version.
\bibliographystyle{imsart-nameyear}  \bibliography{egarch}

\section*{Appendix:   {\bf (MM)} holds in the EGARCH(1,1) model iff {\bf (MM')} holds}

To check {\bf (MM)} is technical as we have to compute explicitly the diagonal terms the matrix $\bfB=\E [\nabla g_{t}(\theta_{0})(\nabla g_{t}(\theta_{0}))^{T} ]$.
Similar computations have been done in \citet{DK2009}. Remember that $\bfV=\bfP^{-1}\bfQ\bfP^{-1}$ with $\bfP=\E[\mathbb H s_0(\theta_0)]$
and $\bfQ=\E[\nabla s_0(\theta_0)\nabla s_0(\theta_0)^T]$. Let us first prove the three identities $\bfP=2^{-1}\bfB$, 
$\bfQ=4^{-1}(\E Z_0^4-1)\bfB$ and thus $\bfV=(\E Z_0^4-1)\bfB^{-1}$. For the first identity, we compute
\begin{align*}
\bfP& =2^{-1}\E\left[(\nabla g_{t}(\theta_{0})(\nabla g_{t}(\theta_{0}))^{T}Z_{0}^{2}+\mathbb Hg_{t}(\theta_{0})(1-Z_{0}^{2})\right]\\
 & =2^{-1}\E [\nabla g_{t}(\theta_{0})(\nabla g_{t}(\theta_{0}))^{T} ]=2^{-1}\bfB.
 \end{align*}
For the second identity, we compute
\begin{align*}
\bfQ& =\E\left[\frac{1}{4}\E\left[\nabla g_{t}(\theta_{0})(\nabla g_{t}(\theta_{0}))^{T}(1-\Z_{t}^{2})^{2}\right]|\cF_{t-1}\right]\\
 & =4^{-1}\E[(1-\Z_{0}^{2})^{2}]\E[\nabla g_{t}(\theta_{0})(\nabla g_{t}(\theta_{0}))^{T}]=4^{-1}(\E Z_{0}^{4}-1) \bfB\end{align*}
and the third identity follows the first ones. Thus, for checking the assumption {\bf (MM)}, it is enough to check that diagonal coefficients $\bfB_{ii}$ are well defined when $\E(Z^4_0)<\infty$.
Let us denote
$W_{t}=\gamma_0 Z_{t}+\delta_0 |Z_{t} |$, $U_{t}=(1,\log\sigma_{t}^{2},Z_{t},|Z_{t}|)$ and $V_{t}=\beta_0 -2^{-1}(\gamma_0 Z_{t}+\delta_0\left|Z_{t}\right|)$.
 Then $(\nabla g_{t}(\theta_{0}))$ is the solution of the linear SRE
$$
\nabla g_{t}(\theta_{0})=U_{t-1}+V_{t-1}\nabla g_{t-1}(\theta_{0})=\sum_{l=1}^{\infty}\left(U_{t-l}\prod_{k=1}^{l-1}V_{t-k}\right).$$Using the convention $\prod_{k=1}^{0}V_{t-k}=1$, we obtain the expression
$$
\nabla g_{t}(\theta_{0})=\sum_{l=1}^{\infty}\left(U_{t-l}\prod_{k=1}^{l-1}V_{t-k}\right).
$$
More precisely, we have the expressions:
then 
\begin{align*}
\bfB_{11}&=\E\left(\frac{\partial g_{t}(\theta_{0})}{\partial\theta_{1}}\right)^{2}=\E\left[\sum_{l=1}^{\infty}\prod_{k=1}^{l-1}V_{t-k}\right]^{2},\\
\bfB_{22}&=\E\left(\frac{\partial g_{t}(\theta_{0})}{\partial\theta_{2}}\right)^{2}=\E\left[\sum_{l=1}^{\infty}\log\sigma_{t-l}^{2}\prod_{k=1}^{l-1}V_{t-k}\right]^{2},\\
\bfB_{33}&=\E\left(\frac{\partial g_{t}(\theta_{0})}{\partial\theta_{3}}\right)^{2}=\E\left[\sum_{l=1}^{\infty}Z_{t-l}\prod_{k=1}^{l-1}V_{t-k}\right]^{2},\\
\bfB_{44}&=\E\left(\frac{\partial g_{t}(\theta_{0})}{\partial\theta_{i}}\right)^{2}=\E\left[\sum_{l=1}^{\infty}\left|Z_{t-l}\right|\prod_{k=1}^{l-1}V_{t-k}\right]^{2}.
\end{align*}
To prove that condition {\bf (MM)} is satisfied, i.e. that $\sum_{i=1}^4\bfB_{ii}<\infty$, we use the following Lemma
\begin{lem}\label{lem:EV2}
$\sum_{i=1}^4\bfB_{ii}<\infty$ iff $\E V_0^2<1$.
\end{lem}
\begin{proof} That the first coefficient $B_{11}$ is finite comes easily from the identities
\begin{align*}
\bfB_{11}=\E(\sum_{l=1}^{\infty}\prod_{k=1}^{l-1}V_{t-k})^{2} & =\E(\sum_{l=1}^{\infty}\sum_{l^{'}=1}^{\infty}\prod_{k=1}^{l-1}V_{t-k}\prod_{k^{'}=1}^{l^{'}-1}V_{t-k^{'}})\\
 & =\E(2\sum_{l\geq1}^{\infty}(\prod_{k=1}^{l-1}V_{t-k})^{2}\sum_{l^{'}>l}^{\infty}\prod_{k^{'}=l}^{l^{'}-1}V_{t-k^{'}}+\E\sum_{l=1}^{\infty}(\prod_{k=1}^{l-1}V_{t-k}{}^{2})\\
 & =2\sum_{l\geq1}^{\infty}(\E V_{0}^{2})^{l-1}\frac{\E V_{0}}{1-\E V_{0}}+\frac{1}{1-\E V_{0}^{2}}\\
 & =2\frac{1}{1-\E V_{0}^{2}}\times\frac{\E V_{0}}{1-\E V_{0}}+\frac{1}{1-\E V_{0}^{2}}.
\end{align*}
For the second coefficient $\bfB_{22}$, it is more complicated. We need some preliminary work.
We know that $W_{t}=\gamma_0 Z_{t}+\delta_0\left|Z_{t}\right|=2(\beta_0-V_{t}),$
and 
$$
\log\sigma_{t}^{2}=\frac{\alpha_0}{1-\beta_0}+\sum_{k=1}^{\infty}\beta_0^{k-1}W_{t-k}=\frac{\alpha_0+2\beta_0}{1-\beta_0}-2\sum_{k=1}^{\infty}\beta^{k-1}V_{t-k}$$
 so, we decompose $\bfB_{22}$ into three parts, 
\begin{align*}
\bfB_{22}=&\E\left[\sum_{l=1}^{\infty}\log\sigma_{t-l}^{2}\prod_{k=1}^{l-1}V_{t-k}\right]^{2}\\
=&\E\left[\sum_{l=1}^{\infty}\left(\frac{\alpha_0+2\beta_0}{1-\beta_0}-2\sum_{k=1}^{\infty}\beta_0^{k-1}V_{t-l-k}\right)\prod_{k^{'}=1}^{l-1}V_{t-k^{'}}\right]^{2}\\
=&(\frac{\alpha_0+2\beta_0}{1-\beta_0})^{2}\E\left[\sum_{l=1}^{\infty}\prod_{k^{'}=1}^{l-1}V_{t-k^{'}}\right]^{2}+4\E\left[\sum_{l=1}^{\infty}\sum_{k=1}^{\infty}\beta_0^{k-1}V_{t-l-k}\prod_{k^{'}=1}^{l-1}V_{t-k^{'}}\right]^{2}\\
&-4\times\frac{\alpha_0+2\beta_0}{1-\beta_0}\E\left[\sum_{l=1}^{\infty}\sum_{k=1}^{\infty}\beta_0^{k-1}V_{t-l-k}\prod_{k^{'}=1}^{l-1}V_{t-k^{'}}\right].
\end{align*}

That the first term of the sum is finite is already known. For the last term, it is straightforward from
$
\E\sum_{l=1}^{\infty}\sum_{k=1}^{\infty}\beta_0^{k-1}V_{t-l-k}\prod_{k^{'}=1}^{l-1}V_{t-k^{'}}=(1-\beta_0)^{-1}\E V_{0}/(1-\E V_{0}).
$
For the second term of the sum, we need an expansion 
\begin{align*}
&\hspace{-1cm}\left[\sum_{l=1}^{\infty}\sum_{k=1}^{\infty}\beta_0^{k-1}V_{t-l-k}\prod_{k^{'}=1}^{l-1}V_{t-k^{'}}\right]^{2}\\
=&2\times\sum_{1\leq l<l^{'}<\infty1}^{\infty}\sum_{p,q=1}^{\infty}\beta_0^{p+q-2}V_{t-l-p}V_{t-l^{'}-q}\prod_{p^{'}=1}^{l-1}V_{t-p^{'}}^{2}\prod_{q^{'}=l}^{l^{'}-1}V_{t-q^{'}}\\
&+\sum_{l=1}^{\infty}\sum_{p,q=1}^{\infty}\beta_0^{p+q-2}V_{t-l-p}V_{t-l-q}\prod_{p^{'}=1}^{l-1}V_{t-p^{'}}^{2}\\
=&4\times\sum_{1\leq l<l^{'}<\infty}\sum_{1\leq p<q<\infty}\beta_0^{p+q-2}V_{t-l-p}V_{t-l^{'}-q}\prod_{p^{'}=1}^{l-1}V_{t-p^{'}}^{2}\prod_{q^{'}=l}^{l^{'}-1}V_{t-q^{'}}\\
&+2\times\sum_{1\leq l<l^{'}<\infty}\sum_{p=1}^{\infty}\beta_0^{2p-2}V_{t-l-p}V_{t-l^{'}-p}\prod_{p^{'}=1}^{l-1}V_{t-p^{'}}^{2}\prod_{q^{'}=l}^{l^{'}-1}V_{t-q^{'}}\\
&+2\sum_{l=1}^{\infty}\sum_{1\leq p<q<\infty}\beta_0^{p+q-2}V_{t-l-p}V_{t-l-q}\prod_{p^{'}=1}^{l-1}V_{t-p^{'}}^{2}+\sum_{l=1}^{\infty}\sum_{p=1}^{\infty}\beta_0^{2p-2}V_{t-l-p}^{2}\prod_{p^{'}=1}^{l-1}V_{t-p^{'}}^{2}
\end{align*}

and in expectation we obtain a bounded term if $\E V_0^2<1$:

\begin{align*}
&\hspace{-1cm}\E\left[\sum_{l=1}^{\infty}\sum_{k=1}^{\infty}\beta_0^{k-1}V_{t-l-k}\prod_{k^{'}=1}^{l-1}V_{t-k^{'}}\right]^{2}\\
=&4\times\frac{\E V_{0}^{2}}{1-\E V_{0}^{2}}\left[\frac{\beta_0}{(1-\beta_0)(1-\beta_0^{2})}\frac{\E V_{0}}{1-\E V_{0}}-\frac{1}{(1-\beta_0)(1-\beta_0^{2})}\frac{\E V_{0}\beta}{1-\beta_0^{2}\E V_{0}}\right]\\
&+4\times\frac{\beta_0(\E V_{0})^{3}}{(1-\beta_0)(1-\beta_0^{2})(1-\beta_0^{2}\E V_{0})}\frac{1}{1-\E V_{0}^{2}}\\
&+2\times\frac{1}{1-\beta_0^{2}}\frac{\E V_{0}^{2}}{1-\E V_{0}^{2}}\left[\frac{\E V_{0}}{1-\E V_{0}}-\frac{\E V_{0}}{1-\beta_0^{2}\E V_{0}}\right]\\
&+2\frac{1}{1-\E V_{0}^{2}}\frac{(\E V_{0})^{3}}{(1-\beta_0^{2})(1-\beta_0^{2}\E V_{0})}\\
&+2\frac{1}{1-\E V_{0}^{2}}(\E V_{0})^{2}\frac{\beta_0}{(1-\beta_0)(1-\beta_0^{2})}+\frac{\E V_{0}^{2}}{1-\E V_{0}^{2}}\frac{1}{1-\beta_0^{2}}.
\end{align*}
That $\bfB_{33}$ is finite under $\E V_0^2<1$ comes from
\begin{align*}
\bfB_{33}=&\E\left[\sum_{l=1}^{\infty}Z_{t-l}\prod_{k=1}^{l-1}V_{t-k}\right]^{2} \\
=&2\E\sum_{l=1}^{\infty}\sum_{l^{'}>l}^{\infty}Z_{t-l}Z_{t-l^{'}}\prod_{k=1}^{l-1}V_{t-k}\prod_{k^{'}=1}^{l^{'}-1}V_{t-k^{'}}
  +\E\sum_{l=1}^{\infty}Z_{t-l}^{2}(\prod_{k=1}^{l-1}V_{t-k})^{2}\\
 =&\E Z_{0}^{2}\sum_{l=1}^{\infty}(\E V_{0}^{2})^{l-1} 
  =\frac{\E Z_{0}^{2}}{1-\E V_{0}^{2}}.
\end{align*}
That  the last coefficient is also finite comes form the computation
\begin{align*}
\bfB_{44}=&\E\left[\sum_{l=1}^{\infty}\left|Z_{t-l}\right|\prod_{k=1}^{l-1}V_{t-k}\right]^{2}\\
=&2\E\sum_{l=1}^{\infty}\sum_{l^{'}>l}^{\infty}\left|Z_{t-l}\right|\left|Z_{t-l^{'}}\right|\prod_{k=1}^{l-1}V_{t-k}\prod_{k^{'}=1}^{l^{'}-1}V_{t-k^{'}}+\E\sum_{l=1}^{\infty}Z_{t-l}^{2}(\prod_{k=1}^{l-1}V_{t-k})^{2}\\
=&2\sum_{l=1}^{\infty}\sum_{l^{'}>l}^{\infty}\E\left|Z_{t-l^{'}}\right|\E\left(\prod_{k=1}^{l-1}V_{t-k}^{2}\right)\E\left(\left|Z_{t-l}\right|\prod_{k^{'}=l}^{l^{'}-1}V_{t-k^{'}}\right)+\frac{\E Z_{0}^{2}}{1-\E V_{0}^{2}}\\
=&2\sum_{l=1}^{\infty}\sum_{l^{'}>l}^{\infty}(\E\left|Z_{0}\right|)(\E V_{0}^{2})^{l-1}(\E\left|Z_{0}\right|V_{0})\E V_{0}^{l^{'}-l-1}+\frac{\E Z_{0}^{2}}{1-\E V_{0}^{2}}\\
=&\frac{2\E\left|Z_{0}\right|(\E\left|Z_{0}\right|V_{0})}{(1-\E V_{0})(1-\E V_{0}^{2})}+\frac{\E Z_{0}^{2}}{1-\E V_{0}^{2}}.
\end{align*}
\vspace{-6cm}
\end{proof}

\end{document}